\RequirePackage{etoolbox}
\csdef{input@path}{{style/}{graphics/}}
\makeatletter
\input{arxiv-vmsta.cfg}
\makeatother
\documentclass[numbers,compress]{vmsta}
\volume{1}
\pubyear{2014}
\firstpage{117}
\lastpage{126}
\doi{10.15559/vmsta-2014.12}

\theoremstyle{plain}
\newtheorem{thm}{Theorem}[section]
\newtheorem{prop}{Proposition}[section]

\theoremstyle{definition}

\newtheorem{ex}{Example}[section]

\newcommand{\rrvert}{\vert}
\newcommand{\llvert}{\vert}
\allowdisplaybreaks

\begin{document}

\renewcommand{\Re}{\mathbb{R}}
\newcommand{\Ff}{\mathcal{F}}
\newcommand{\prt}{\partial}

\begin{frontmatter}

\title{Rates of approximation of nonsmooth integral-type functionals
of Markov processes}

\author[a]{\inits{Iu.}\fnm{Iu.}\snm{Ganychenko}}\email{iurii\_ganychenko@ukr.net}
\author[b]{\inits{A.}\fnm{A.}\snm{Kulik}\corref{cor1}}\email{kulik.alex.m@gmail.com}
\cortext[cor1]{Corresponding author.}
\address[a]{Taras Shevchenko National University of Kyiv, Kyiv, Ukraine}
\address[b]{Institute of Mathematics, National Academy of Sciences of Ukraine, Kyiv,
Ukraine}
\markboth{Iu. Ganychenko, A. Kulik}{Rates of approximation of
non-smooth integral type functionals of Markov processes}


\begin{abstract}
We provide strong $L_p$-rates of approximation of nonsmooth integral-type functionals of Markov processes by integral sums. Our approach is, in a sense,
process insensitive and is based on a modification of some
well-developed estimates from the theory of continuous additive
functionals of Markov processes.
\end{abstract}


\received{5 April 2014}
\revised{28 May 2014}
\accepted{27 August 2014}
\publishedonline{15 September 2014}
\end{frontmatter}

\section{Introduction}

Let $X_t$, $t\geq0$, be an  $\mathbb{R}^d$-valued Markov process. We study an integral functional
\[
I_T(h)=\int_0^Th(X_t)\, dt
\]
of this process. The most natural numerical scheme to approximate such
a~functional is the sequence of integral sums
\[
I_{T,n}(h)={T\over n}\sum_{k=0}^{n-1}h(X_{(kT)/n}), \quad n\geq1,
\]
and the main objective of this paper is to study approximation rates
within this scheme. The function $h$, in general, is not assumed to be
smooth, and therefore the mapping
\[
\bigl\{x_t, t\in[0,T] \bigr\}\mapsto\int_0^Th(x_t)\, dt
\]
may fail to be Lipschitz continuous (and even simply continuous) on a
natural functional space of the trajectories of $X$ (e.g., $C(0,T)$ or
$D(0,T)$). This makes it impossible to carry out the error analysis
with a classical technique (see, e.g., \cite{KloedenPlaten}). The typical case
of  interest here is $h=1_A$, with $I_T(h)$ being respectively the
occupation time of $X$ at the set $A$ up to the time moment $T$.

In the paper, we establish strong $L_p$-approximation rates, that is, the
bounds for
\[
E \bigl|I_{T}(h)-I_{T,n}(h) \bigr|^p.
\]
Our research is strongly motivated by the recent paper \cite
{Kohatsu-Higa}, where such a problem was studied in a particularly
important case where  $X$ is a one-dimensional diffusion, and we refer
the reader to \cite{Kohatsu-Higa} for more motivation and background on
the subject. The technique developed in \cite{Kohatsu-Higa}, involving
both the Malliavin calculus tools and the Gaussian bounds for the
transition probability density, relies substantially on the structure
of the process, and hence it seems not easy to extend this approach to
other classes of processes, for example, multidimensional diffusions or
solutions to L\'evy driven SDEs.

We would like to explain in this note  that, in order to get the
required approximation rates, one can modify some well-developed
estimates from the theory of continuous additive functionals of Markov
processes. An advantage of such approach is that the assumptions on the
process are formulated only in  terms of its  transition
probability density and therefore are quite flexible. The basis for
the approach is given by the fact that the weak approximation rates for
\[
EI_{T}(h)-EI_{T,n}(h)
\]
are available as a consequence of a bound for the derivative w.r.t. $t$
of the transition probability density; see \cite{Gobet}, Theorem 2.5,
and Proposition \ref{p21} below. To explain the principal idea of the
approach, let us assume for a while that $h$ is nonnegative and
bounded. Then the integral functional $I_T(h)$ is a \emph
{$W$-functional} of the process $X$; see \cite{Dynkin}, Chapter 6. It
is well known that the properties of a $W$-functional are mainly
controlled by its \emph{characteristic}, that is, the expectation
\[
E_xI_T(h).
\]
In particular, the convergence of characteristics implies the
$L_2$-convergence of the respective functionals. The core of our approach
is that we extend the Dynkin's technique for a study of convergence of
$W$-functionals and give approximation rates for integral functionals
$I_{T}(h)$ by difference functionals $I_{T,n}(h)$, based on the weak
approximation rates for their expectations. We remark that now we are
beyond the scopes of the original Dynkin's theory because $I_{T}(h)$
\emph{may fail} to be a $W$-functional (we do not assume $h$ to be
nonnegative), and $I_{T,n}(h)$ \emph{definitely fails} to be a
$W$-functional. In addition, Dynkin's theory addresses $L_2$-bounds, whereas, in general, we are interested in $L_p$-bounds. This brings some
extra difficulties, which however are not really substantial, and we
resolve them in a way similar to the one used in the classical
Khas'minskii lemma; see, for example, Lemma 2.1 in \cite{Sznitman}.

\section{Main results}\begingroup\abovedisplayskip=7pt\belowdisplayskip=7pt
\subsection{Notation, assumptions, and auxiliaries}
In what follows, $P_x$ denotes the law of the Markov process with
$X_0=x$, and $E_x$ denotes the expectation w.r.t. this law. The natural
filtration of the process $X$ is denoted by $\{\Ff_t, t\geq0\}$. The
process $X$ is assumed to possess a transition probability density,
denoted below by $p_t(x,y)$. By $C$ we denote a generic constant; the
value of $C$ may vary from place to place. Both the absolute value of a
real number and the Euclidean norm in $\Re^d$ are denoted by $|\cdot|$.

Our standing assumption on the process $X$ under investigation is the following.
\begin{enumerate}
\item[\textbf{X.}] The transition probability density $p_t(x,y)$ is
differentiable w.r.t. $t$ and satisfies
\begin{align}
\label{dens_bound}
p_t(x,y)
&{}\leq C_Tt^{-d/\alpha} Q \bigl(t^{-1/\alpha}(x-y) \bigr), \quad t\leq T,\\
\label{der_bound}
\bigl|\prt_tp_t(x,y)\bigr|
&{}\leq C_Tt^{-1-d/\alpha}
Q \bigl(t^{-1/\alpha}(x-y) \bigr), \quad t\leq T,
\end{align}
with some fixed $\alpha\in(0,2]$ and distribution density $Q$.
\end{enumerate}

The assumption \textbf{X} is motivated by the following class of
processes of particular interest.

\begin{ex}
Let $X$ be a symmetric $\alpha$-stable process with $\alpha
\in(0,2]$; in the case $\alpha=2$ this is just a Brownian motion. Then
\[
p_t(x,y)=t^{-d/\alpha}g^{(\alpha)} \bigl(t^{-1/\alpha}(x-y)
\bigr)
\]
with $g^{(\alpha)}$ being the distribution density of $X_1$.
Respectively, (\ref{der_bound}) holds  with $Q=Q_\alpha$,
\[
Q_\alpha(x)= %
\begin{cases}
c_1e^{-c_2|x|^2},         & \alpha=2,\\
{c\over1+|x|^{d+\alpha}}, &\alpha\in(0,2),
\end{cases} %
\]
where $c_2<(2EX_1^2)^{-1}$ and $c_1, c$ should be chosen such that
$\int_{\Re^d}Q(x)\, dx=1$.

Observe that, in a sense, this bound is ``stable under perturbations
of the process $X$.'' Namely, if $X$ is a uniformly elliptic diffusion
with H\"older continuous coefficients, then (\ref{der_bound}) with
$Q=Q_2$ and properly chosen $c_2$ is provided by the classical \emph
{parametrix method}; see \cite{Fr64}. An analogue of the parametrix
method for $\alpha$-stable generators with state-dependent coefficients
yields the bound (\ref{der_bound}) with $Q=Q_{\alpha'}$, $\alpha'<\alpha$,
for $\alpha$-stable driven processes $X$; see \cite{EIK04} and \cite{Ko89}.
\end{ex}

Our principal assumption on the function $h$ is the following.
\begin{enumerate}
\item[\textbf{H1.}]
The function $h$ satisfies
\[
\sup_{x}{|h(x)|\over V(|x|)}<\infty,
\]
where $V:\Re^+\to[1, +\infty)$ is a fixed function such that
\begin{itemize}
\item$\int_{\Re^d}V(T|x|)Q(x)\, dx<\infty$, $T\geq0$;
\item$V(r_1)\leq V(r_2)$, $r_1\leq r_2$;
\item$V$ is submultiplicative, that is,
\[
V(r_1+r_2)\leq V(r_1)V(r_2),
\quad r_1,r_2\in\Re^+.
\]
\end{itemize}
\end{enumerate}
Observe that for a bounded $h$, condition \textbf{H1} holds
trivially with $V\equiv1$. On the other hand, in particular cases, one
can weaken the assumptions on $h$ by using nontrivial ``weight
functions'' $V$. For instance, if $Q=Q_\alpha$ from the above example,
then one can take
\[
V(r)= %
\begin{cases}e^{Cr},& \alpha=2,\\
(1+r)^\beta,        &\alpha\in(0,2),
\end{cases}\quad  r \in\Re^+,
\]
with arbitrary $C$ and $\beta\in(0, \alpha)$. We denote
\[
\|h\|_V=\sup_{x}{|h(x)|\over V(|x|)}.
\]
The following auxiliary statement is crucial for the whole approach.
Its proof is completely analogous to the proof of (a part of) Theorem
2.5 \cite{Gobet}, but in order to make the exposition self-sufficient,
we give it here.

\begin{prop}\label{p21} Let \textbf{X} and \textbf{H1} hold. Then
\[
\bigl|E_xI_{T}(h)-E_xI_{T,n}(h) \bigr|
\leq
\biggl({\log{n}\over n} \biggr)C_{T,Q,V}\|h\|_V V\bigl(|x|\bigr)
\]
with
\[
C_{T,Q,V}=T \max \biggl\{C_T \biggl(\int_{\Re^d}Q(y)V
\bigl(T^{1/\alpha
}|y| \bigr)\, dy \biggr), 1 \biggr\}.
\]
\end{prop}
\begin{proof} Write
\begin{align*}
E_xI_{T}(h)-E_xI_{T,n}(h)
&{} = \int_{0}^{T/n}\int_{\Re^d}p_t(x,y)h(y)\,dydt-\frac{T}{n}h(x)\\
&\quad{}+
\sum_{k=2}^n\int_{(k-1)T/n}^{kT/n}\int_{\Re^d}
\bigl(p_t(x,y)-p_{(k-1)T/n}(x,y)\bigr)h(y)\,dydt.
\end{align*}
We have, by the bound for $p_t(x,y)$ in (\ref{dens_bound}) and
properties of $V$,
\begin{align}
\hspace*{-1pt}\biggl\llvert\!\int_{0}^{T/n}\!\int_{\Re^d}\!p_t(x,y)h(y)\,dydt\biggr\rrvert
&{}\leq
C_T\|h\|_V\!\int_{0}^{T/n}\!\!\!t^{-d/\alpha}\!\!\int_{\Re^d}\!Q \bigl(t^{-1/\alpha}(x-y)\bigr)V(|y|)\, dydt
\nonumber\quad\mbox{}\\
&{}=
C_T\|h\|_V\int_{0}^{T/n}\int_{\Re^d}Q(z)V \bigl(|x+t^{1/\alpha}z| \bigr)\, dzdt
\nonumber\\
&{}\leq
n^{-1}T C_T\|h\|_V V\bigl(|x|\bigr)\int_{\Re^d}Q(z)V \bigl(T^{1/\alpha}|z| \bigr)\, dz.
\label{11} %
\end{align} %
Next,
\[
\sum_{k=2}^n\int_{(k-1)T/n}^{kT/n}
\int_{\Re^d} \bigl(p_t(x,y)-p_{(k-1)T/n}(x,y)
\bigr)h(y)\,dydt=\int_{\Re
^d}K_{n,T}(x,y)h(y)\, dy
\]
with
\[
\begin{aligned}K_{n,T}(x,y)&=\sum
_{k=2}^n \int_{(k-1)T/n}^{kT/n}
\int_{(k-1)T/n}^t \prt_sp_s(x,y)
\,dsdt
\\
&=\sum_{k=2}^n\int_{(k-1)T/n}^{kT/n}
\biggl({kT\over n}-s \biggr)\prt_sp_s(x,y)
\,ds. \end{aligned} %
\]
Then
\[
\bigl|K_{n,T}(x,y)\bigr|
\leq
{T\over n}
\int_{T/n}^T\bigl|\prt_sp_s(x,y)\bigr|\,ds,
\]
and therefore, using the bound for $\prt_tp_t(x,y)$ in (\ref
{der_bound}), we obtain, similarly to (\ref{11}),
\begin{align*}
\biggl\llvert\int_{\Re^d}\!\!K_{n,T}(x,y)h(y)
\, dy \biggr\rrvert
&{}\leq
{T\over n}C_T\|h\|_V\!\!\int_{T/n}^T\!\!s^{-1}\!\! \int_{\Re^d}\!s^{-d/\alpha}Q \bigl(s^{-1/\alpha}(x-y)\bigr)V\bigl(|y|\bigr)\, dyds
\\
&{}\leq
{T\over n}C_T\|h\|_V V\bigl(|x|\bigr) \biggl(\int_{\Re^d}\!Q(y)V \bigl(T^{1/\alpha}|y| \bigr)\, dy
\biggr)\int_{T/n}^T\!s^{-1}\, ds
\\
&{}=
{T\over n} (\log{n})C_T\|h\|_V V\bigl(|x|\bigr)\biggl(\int_{\Re^d}Q(y)V \bigl(T^{1/\alpha}|y| \bigr)\, dy\biggr),
\end{align*}
which completes the proof.
\end{proof}\endgroup

\subsection{Approximation rate in  terms of \texorpdfstring{$\|h\|_V$}{$||h||_V$}}

Our main estimate, in a shortest and most transparent form, is
presented in the following theorem, which concerns the case where the
only assumption on~$h$ is that the weighted sup-norm $\|h\|_V$ is finite.

\begin{thm}\label{t1} Let \textbf{X} and \textbf{H1} hold. Then
for every $p\geq2$ such that
\[
\int_{\Re^d}V^p \bigl(T^{1/\alpha}|x|
\bigr)Q(x) \, dx<\infty,
\]
we have
\[
E_x \bigl|I_{T}(h)-I_{T,n}(h) \bigr|^p\leq C
\biggl({\log n\over n} \biggr)^{p/2}\|h\|^p_V
V^p\bigl(|x|\bigr)
\]
with constant $C$ depending on $T,Q,V,p$ only.
\end{thm}
\begin{proof} Denote, for $t\in[kT/n, (k+1)T/n)$,
\[
\eta_n(t)={kT\over n}, \qquad\zeta_n(t)=
{(k+1)T\over n},
\]
and write the difference $I_{t}(h)-I_{t,n}(h)$ in the integral
form:
\[
J_{t,n}(h):=I_{t}(h)-I_{t,n}(h)=\int
_0^t \Delta_n(s)ds, \quad\Delta
_n(s):=h(X_s) - h(X_{\eta_n (s)}).
\]
Hence, this difference is an absolutely continuous function of $t$, and
using the Newton--Leibnitz formula twice, we get
\[
\bigl|J_{T,n}(h)\bigr|^p=p(p-1)\int_0^T
\bigl|J_{s,n}(h)\bigr|^{p-2}\Delta_n(s)
\biggl(\int_s^T\Delta_n(t)\, dt \biggr)ds.
\]
We then write
\[
\bigl|J_{T,n}(h)\bigr|^p=p(p-1)
\bigl(H_{T,n,p}^1(h)+H_{T,n,p}^2(h)\bigr)\leq p(p-1)
\bigl(\tilde H_{T,n,p}^1(h)+H_{T,n,p}^2(h)\bigr),
\]
where
\begin{align*}
H_{T,n,p}^1(h)
&{}=
\int_0^T
\bigl|J_{s,n}(h)\bigr|^{p-2}\Delta_n(s)
\biggl(\int_s^{\zeta_n(s)}\Delta_n(t)\, dt \biggr)ds,\\
\tilde H_{T,n,p}^1(h)
&{}=\int_0^T
\bigl|J_{s,n}(h)\bigr|^{p-2}
\bigl|\Delta_n(s)\bigr|
\biggl\llvert\int_s^{\zeta_n(s)}\Delta_n(t)\, dt \biggr\rrvert
ds,\\
H_{T,n,p}^2(h)
&{}=\int_0^T
\bigl|J_{s,n}(h)\bigr|^{p-2}\Delta_n(s)
\biggl(\int_{\zeta_n(s)}^T\Delta_n(t)\, dt \biggr)ds.
\end{align*}
Let us estimate separately the expectations of $\tilde H_{T,n,p}^1(h)$ and
$H_{T,n,p}^2(h)$. By the H\"older inequality,
\begin{align*}
E_x\tilde H_{T,n,p}^1(h)
&{} \leq
\biggl(E_x\int_0^T
\bigl|J_{s,n}(h)\bigr|^{p}\, ds \biggr)^{1-2/p}\\
&\quad{}\times
\biggl(E_x\int_0^T
\bigl|\Delta_n(s)\bigr|^{p/2} \biggl\llvert\int_s^{\zeta_n(s)}
\Delta_n(t)\, dt \biggr\rrvert^{p/2}\, ds
\biggr)^{2/p}.
\end{align*}
Again by the H\"older inequality,
\begin{align*}
&E_x\int_0^T
\bigl|\Delta_n(s)\bigr|^{p/2}
\biggl\llvert\int_s^{\zeta_n(s)}\Delta_n(t)\,
dt \biggr\rrvert^{p/2}\, ds\\
&\quad{}\leq \biggl({T\over n}
\biggr)^{p/2-1}\int_0^T\int
_s^{\zeta_n(s)}E_x
\bigl|\Delta_n(s)\bigr|^{p/2}
\bigl|\Delta_n(t)\bigr|^{p/2}\, dtds.
\end{align*}
Because $t\in[s, \zeta_n(s)]$, we have $\eta_n(t)=\eta_n(s)$, and,
consequently,
\begin{align}
&{}E_x
\bigl|\Delta_n(s)\bigr|^{p/2}
\bigl|\Delta_n(t)\bigr|^{p/2}\nonumber\\
&\quad{} =
E_x
\bigl|h(X_s)-h(X_{\eta_n(s)})\bigr|^{p/2}
\bigl|h(X_t)-h(X_{\eta_n(s)})\bigr|^{p/2}
\nonumber\\
&\quad{}\leq
\bigl(E_x
\bigl|h(X_s)-h(X_{\eta_n(s)})\bigr|^{p}
\bigr)^{1/2}
\bigl(E_x
\bigl|h(X_t)-h(X_{\eta_n(s)})\bigr|^{p}
\bigr)^{1/2}
\nonumber\\
&\quad{}\leq
\|h\|_V^p2^{p-1}
\bigl(E_x
\bigl(V^p\bigl(|X_s|\bigr)+V^p\bigl(|X_{\eta_n(s)}|\bigr)
\bigr)
\bigr)^{1/2}\nonumber\\
&\qquad{}\times
\bigl(E_x
\bigl(V^p\bigl(|X_t|\bigr)+V^p\bigl(|X_{\eta_n(s)}|\bigr)
\bigr)
\bigr)^{1/2}.
\label{13}
\end{align} %
 By the properties of $V$ and the bound (\ref{dens_bound}) we have
\begin{align*}
E_xV^{p}\bigl(|X_r|\bigr)
&{}=
\int_{\Re^d}p_{r}(x,y)V^{p}\bigl(|y|\bigr)\,dy\\
&{}
\leq
C_Tr^{-d/\alpha}\int_{\Re^d}Q
\bigl(r^{-1/\alpha}(x-y) \bigr)V^{p}\bigl(|y|\bigr)\, dy
\\
&{}\leq C_TV^{p}\bigl(|x|\bigr)
\biggl(\int_{\Re^d}Q(y)V^{p}\bigl(T^{1/\alpha}|y| \bigr)\, dy \biggr)
\end{align*} %
for any $r\in(0, T]$. Using this bound with $r=t,s, \eta_n(s)$ and
recalling that
$|\zeta_n(s)-s|\leq 1/n$, we get
\[
E_x\tilde H_{T,n,p}^1(h)\leq C
\biggl(E_x\int_0^T
\bigl|J_{s,n}(h)\bigr|^{p}\, ds \biggr)^{1-2/p}
\biggl({1\over n} \biggr)\|h\|_V^2V^{2}\bigl(|x|\bigr)
\]
with constant $C$ depending on $T,Q,V,p$ only.

Next, observe that, for every $s$, the variables
\[
\Delta_n(s), \qquad|J_{s,n}(h)|^{p-2}
\Delta_n(s)
\]
are $\Ff_{\zeta_n(s)}$-measurable. Hence,
\begin{align*}
E_xH_{T,n,p}^2(h)
&{} =E_x
\biggl(\int_0^T
\bigl|J_{s,n}(h)\bigr|^{p-2} \Delta_n(s)E_x
\biggl(\int_{\zeta_n(s)}^T \Delta_n(t)\, dt \big|
\Ff_{\zeta_n(s)}
\biggr)ds
\biggr)
\\
&{}\leq
E_x
\biggl(\int_0^T
\bigl|J_{s,n}(h)\bigr|^{p-2}
\bigl|\Delta_n(s)\bigr|
\biggl\llvert E_x
\biggl(\int_{\zeta_n(s)}^T\Delta_n(t)\, dt
\big|\Ff_{\zeta_n(s)}
\biggr)
\biggr\rrvert ds
\biggr).
\end{align*}
By Proposition \ref{p21} and the Markov property of $X$ we have
\[
\biggl\llvert E_x \biggl(
\int_{\zeta_n(s)}^T\Delta_n(t)\, dt
\big|
\Ff_{\zeta_n(s)}
\biggr)
\biggr\rrvert
\leq
C \biggl({\log n\over n} \biggr)
\|h\|_VV
\bigl(|X_{\zeta_n(s)}|\bigr).
\]
Hence, again, using the H\"older inequality we get
\begin{align*}
E_xH_{T,n,p}^2(h)
&{}\leq C \biggl( {\log n\over n} \biggr)\|h\|_V
\biggl(E_x\int_0^T
\bigl|J_{s,n}(h)\bigr|^{p}\, ds
\biggr)^{1-2/p}
\\
&\quad{}\times
\biggl(E_x\int_0^T
\bigl|\Delta_n(s)\bigr|^{p/2}V^{p/2}
\bigl(|X_{\zeta_n(s)}|\bigr)\,ds
\biggr)^{2/p}.
\end{align*}
Similarly to (\ref{13}), we have
\[
E_x
\bigl|\Delta_n(s)\bigr|^{p/2}V^{p/2}
\bigl(|X_{\zeta_n(s)}|\bigr)
\leq
C V^{p}\bigl(|x|\bigr)\|h\|^{p/2}_V.
\]
Hence, the above bounds for $E_x\tilde H_{T,n,p}^1(h)$ and
$E_xH_{T,n,p}^2(h)$  finally yield
%
\begin{equation}
\label{14}
E_x
\bigl|J_{T,n}(h)\bigr|^{p}
\leq
C
\biggl(
E_x\int_0^T
\bigl|J_{s,n}(h)\bigr|^{p}\,ds
\biggr)^{1-2/p}
\biggl(
{\log n\over n}
\biggr)
\|h\|_V^2V^{2}
\bigl(|x|\bigr)
\end{equation}
with a constant $C$ depending on $T,Q,V,p$ only. It can be seen easily
that in this inequality one can write  arbitrary $t\leq T$ instead
of $T$, with the same constant $C$. Taking the integral over $t$, we get
\[
E_x\int_0^T
\bigl|J_{t,n}(h)\bigr|^{p}\, dt
\leq
CT
\biggl(
E_x\int_0^T
\bigl|J_{s,n}(h)\bigr|^{p}\,ds
\biggr)^{1-2/p}
\biggl(
{\log n\over n}
\biggr)\| h\|_V^2V^{2}
\bigl(|x|\bigr).
\]
Because $\|h\|_V<\infty$ and $V^p$ satisfies the integrability
condition from the condition of the theorem, the left-hand side
expression in the last inequality is finite. Hence, resolving this
inequality, we get
\[
E_x\int_0^T
\bigl|J_{s,n}(h)\bigr|^{p}\,ds
\leq
(CT)^{p/2}
\biggl(
{\log n\over n}
\biggr)^{p/2}\|h\|_V^pV^{p}
\bigl(|x|\bigr),
\]
which, together with (\ref{14}), gives the required statement.
\end{proof}

\subsection{An improved approximation rate for a H\"older continuous $h$}\begingroup\abovedisplayskip=7pt\belowdisplayskip=7pt

In this section, we consider the case where $h$ has the following
additional regularity property.
\begin{enumerate}
\item[\textbf{H2.}]
The function is H\"older continuous with  index
$\gamma
\in(0,1]$, that is,
\[
\|h\|_\gamma:=\sup_{x\not=y}{|h(x)-h(y)|\over|x-y|^\gamma}<
\infty.
\]
\end{enumerate}
An additional regularity of $h$ allows one to improve the accuracy of
the previous estimates. Namely, the following statement holds.
\begin{thm}\label{t2}
Assume that \textbf{X}, \textbf{H1}, and
\textbf{H2} hold. Then, for every $p\geq2$ such that
\[
\int_{\Re^d}|x|^{\gamma p} Q(x)\, dx<\infty
\]
and
\[
\int_{\Re^d}V^{p/2}
\bigl(T^{1/\alpha}|x|\bigr)Q(x) \, dx<\infty,
\]
we have
\[
E_x
\bigl|I_{T}(h)-I_{T,n}(h)\bigr|^p
\leq
C
\biggl({\log n\over n} \biggr)^{p/2}
n^{-{(\gamma p)/(2\alpha)}}
\|h\|_{\gamma}^{p/2}
\bigl(\|h\|_{\gamma}^{p/2}+\|h\|^{p/2}_V V^{p/2}
\bigl(|x|
\bigr)
\bigr)
\]
with constant $C$ depending on $T,Q,V,p,\gamma$ only.
\end{thm}

\begin{proof} The method of the proof remains the same as that of
Theorem \ref{t1}; hence, we use the same notation. The only new point  is that, instead of the bound
\[
E_x
\bigl|\Delta_n(s)\bigr|^p
\leq
\|h\|_V^p E_x
\bigl(V
\bigl(|X_s|
\bigr)+V
\bigl(
|X_{\eta_n(s)}|
\bigr)
\bigr)^p,
\]
now a more precise inequality is available, based on the H\"older
continuity of~$h$. Namely, we have
\begin{align}
E_x
\bigl|\Delta_n(s)\bigr|^p
&{} =
E_x
\bigl|h(X_s)-h(X_{\eta_n(s)})\bigr|^p\nonumber\\
&{}\leq
\|h\|^p_{\gamma}E_x
\bigl|X_s-X_{\eta_n(s)}\bigr|^{\gamma p}
\leq
C\|h\|^p_{\gamma}
\bigl|s-\eta _n(s)\bigr|^{\gamma p/\alpha},
\label{Lip_bound}
\end{align}
where $C$ depends on $T,Q,p,\gamma$ only.

The last inequality holds due to the following representation. By the
Markov property of $X$, for $r<s$, we have
\[
E_x|X_s-X_r|^{\gamma p}=E_xf(X_r),
\]
where
\begin{align*}
f(z)&{}=
\int_{\mathbb{R}^d}p_{s-r}(z,y)|y-z|^{\gamma p}
\,dy\\
&{}
\leq C_T\int_{\mathbb{R}^d}(s-r)^{-d/\alpha}Q
\bigl((s-r)^{-1/\alpha}(z-y) \bigr)|z-y|^{\gamma p}\,dy
\\
&{} = C_T(s-r)^{\gamma p/\alpha}\int_{\mathbb{R}^d}Q(y)|y|^{\gamma p}
\,dy.
\end{align*}

Thus, for $t\in[s, \zeta_n(s)]$, we have
\begin{align*}
&{}E_x
\bigl|\Delta_n(s)\bigr|^{p/2}
\bigl|\Delta_n(t)\bigr|^{p/2}\nonumber\\
&\quad{}=
E_x
\bigl|h(X_s)-h(X_{\eta_n(s)})\bigr|^{p/2}
\bigl|h(X_t)-h(X_{\eta_n(s)})\bigr|^{p/2}
\\
&\quad{}\leq
\bigl(E_x
\bigl|h(X_s)-h(X_{\eta_n(s)})\bigl|^{p}
\bigr)^{1/2}
\bigl(E_x
\bigl|h(X_t)-h(X_{\eta_n(s)})\bigl|^{p}
\bigr)^{1/2}
\\
&\quad{}\leq
C\|h\|^p_{\gamma}
\bigl|s-\eta_n(s)\bigr|^{\gamma p/(2\alpha)}
\bigl|t-\eta_n(s)\bigr|^{\gamma p/(2\alpha)}\\
&\quad{}\leq
CT\|h\|^p_{\gamma}n^{-(\gamma p)/\alpha}
\end{align*}
and
\[
E_x\tilde H_{T,n,p}^1(h)
\leq
C
\biggl(E_x\int_0^T
\bigl|J_{s,n}(h)\bigr|^{p}\,ds
\biggr)^{1-2/p}
\biggl(
{1\over n}
\biggr)^{1+2\gamma/\alpha}
\|h\| _{\gamma}^2
\]
with  constant $C$ depending on $T,Q,p,\gamma$ only.

Next, using (\ref{Lip_bound}) and (\ref{13}), we have
\[
E_x
\bigl|\Delta_n(s)\bigr|^{p/2}V^{p/2}
\bigl(|X_{\zeta_n(s)}|\bigr)
\leq
C \|h\|^{p/2}_{\gamma}n^{-(\gamma p)/(2\alpha)} V^{p/2}
\bigl(|x|\bigr)
\]
and
\begin{align*}
E_xH_{T,n,p}^2(h)
\leq
C
\biggl(
E_x\int_0^T
\bigl|J_{s,n}(h)\bigr|^{p}\,ds
\biggr)^{1-2/p}
\biggl(
{\log n\over n}
\biggr)n^{-\gamma/\alpha}
\|h\|_V
\|h\|_{\gamma}V
\bigl(|x|\bigr).
\end{align*}

Hence, the previous bounds for $E_x\tilde H_{T,n,p}^1(h)$ and
$E_xH_{T,n,p}^2(h)$ finally yield
\begin{align}
&{}
E_x
\bigl|J_{T,n}(h)\bigr|^{p}\nonumber\\
&\quad{}\leq
C
\biggl(
E_x\int_0^T
\bigl|J_{s,n}(h)\bigr|^{p}\, ds
\biggr)^{1-2/p}
\biggl(
{\log n\over n}
\biggr)n^{-\gamma/\alpha}
\|h\|_{\gamma}
\bigl(
\|h\|_{\gamma}+\|h\|_VV
\bigr(|x|
\bigr)
\bigr)\label{15}
\end{align}
with  constant $C$ depending on $T,Q,V,p,\gamma$ only. Using the same
procedure as that provided at the end of the proof of Theorem 2.1, we
simply have
\begin{align*}
&{}E_x\int_0^T
\bigl|J_{s,n}(h)\bigr|^{p}\,ds
\nonumber\\
&\quad{}
\leq
(CT)^{p/2}
\biggl(
{\log n\over n}
\biggr)^{p/2}
n^{-(\gamma p)/(2\alpha)}
\|h\|_{\gamma}^{p/2}
\bigl(
\|h\|_{\gamma}^{p/2}+\|h\|_V^{p/2}V^{p/2}
\bigl(|x|
\bigr)
\bigr),
\end{align*}
which, together with (\ref{15}), completes the proof.
\end{proof}\endgroup

\subsection{Discussion}

The results of Theorem \ref{t1} and Theorem \ref{t2} should be compared
with Theorem 2.3 in \cite{Kohatsu-Higa}, where, in our notation, the
following bounds were obtained:
\[
E_x
\bigl|
I_T(h)-I_{T,n}(h)
\bigr|^p
\leq
\begin{cases}
Cn^{-p(1+\gamma)/2},&\gamma\in(0,1),\\
C (\log n )^pn^{-p(1+\gamma)/2},&\gamma=1.
\end{cases} %
\]
In Theorem 2.3 of \cite{Kohatsu-Higa}, $h$ satisfies our conditions
\textbf
{H1} and \textbf{H2} with $V(|x|)=e^{C|x|}$, and $X$ is a one-dimensional
diffusion with coefficients that satisfy some smoothness condition
(assumption (H)); in particular, \textbf{X} holds  with $\alpha=2$.
In this case, our bound
\[
C (\log n )^pn^{-p/2-(p\gamma)/4},
\]
given in Theorem \ref{t2}, is somewhat worse. On the other hand, this
bound is of an independent interest because of a wider class of
processes $X$ it applies to.

\section*{Acknowledgements}
The authors are grateful to A. Kohatsu-Higa
for turning their attention to this problem and for helpful
discussions. The first author was partially supported by the Leonard
Euler program, DAAD project No. 57044593.

%

%
\end{document}